\documentclass{birkjour}
 \newtheorem{thm}{Theorem}[section]
 
 \newtheorem{lem}[thm]{Lemma}
 \newtheorem{prop}[thm]{Proposition}
 
 \newtheorem{rem}[thm]{Remark}
 \newtheorem{nota}[thm]{Notation}
 \newtheorem{App}[thm]{Application}
 
 \numberwithin{equation}{section}

\begin{document}

%
%
%
%
%
%
%
%
%

\author[Youssef Aserrar  and  Elhoucien Elqorachi]{Youssef Aserrar  and  Elhoucien Elqorachi}

\address{%
	Ibn Zohr University, Faculty of sciences, 
Department of mathematics,\\
Agadir,
Morocco}

\email{youssefaserrar05@gmail.com, elqorachi@hotmail.com }



\date{January 1, 2020}
\title[Cosine and Sine addition and subtraction law with an automorphism]{Cosine and Sine addition and subtraction law with  an automorphism}
 
\begin{abstract}
Let $S$ be a semigroup. Our main  results is that we describe the complex-valued solutions of the following functional equations 
\[g(x\sigma (y)) = g(x)g(y)+f(x)f(y),\  x,y\in S,\]
\[f(x\sigma (y)) = f(x)g(y)+f(y)g(x),\  x,y\in S,\] 
and 
 \[f(x\sigma (y)) = f(x)g(y)-f(y)g(x),\  x,y\in S,\] 
where $\sigma :S\rightarrow S$ is an automorphism that need not be involutive. As a consequence we show that the first two equations are equivalent to their variants. We also give some applications.
\end{abstract}

\maketitle

\section{Introduction}
Let $S$ be a semigroup and $\sigma :S\rightarrow S$ an automorphism, i.e  $\sigma (xy)=\sigma(x)\sigma(y)$ for all $x,y\in S$. The cosine subtraction formula, sine addition formula and the sine subtraction formula are respectively the functional equations
\begin{equation}
g(x\sigma (y)) = g(x)g(y)+f(x)f(y),\quad x,y\in S,
\label{E3}
\end{equation}
\begin{equation}
f(x\sigma (y)) = f(x)g(y)+f(y)g(x),\quad x,y\in S,
\label{E1}
\end{equation}
\begin{equation}
f(x\sigma (y)) = f(x)g(y)-f(y)g(x),\quad x,y\in S,
\label{E2}
\end{equation}
for unknown functions $f,g:S\rightarrow\mathbb{C}$. These functional equations generalizes respectively the trigonometric identities
\[\cos \left(x-y \right)=\cos (x)\cos (y)+\sin (x)\sin (y),\quad x,y\in \mathbb{R}, \]
\[\sin \left(x+y \right)=\sin (x)\cos (y)+\sin (y)\cos (x),\quad x,y\in \mathbb{R}, \]
\[\sin \left(x-y \right)=\sin (x)\cos (y)-\sin (y)\cos (x),\quad x,y\in \mathbb{R}, \]
and have been investigated by many authors. In the case of an involutive automorphism $\sigma$ , i.e an automorphism satisfying $\sigma(\sigma(x))=x$ for all $x\in S$, equation \eqref{E3} was solved on abelian groups by Vincze \cite{V}, and on general groups by Chung et al. \cite{Ch}. Poulsen and Stetk\ae r \cite{Pou} described the continuous solutions of \eqref{E3}, \eqref{E1} and \eqref{E2} on topological groups. In \cite{Ajb1,Ajb2} Ajebbar and Elqorachi obtained the solutions of \eqref{E3}, \eqref{E1} and \eqref{E2} on semigroups generated by their squares. The solutions of \eqref{E1} with $\sigma=id$ are also described in \cite[Theorem 3.1]{EB2} on a semigroup not necessarily generated by its squares. In \cite{EB1}, Ebanks solved \eqref{E3} and \eqref{E2} on general monoids. Recently the authors \cite{Ase1,Ase2} solved \eqref{E3}, \eqref{E1} and \eqref{E2} on semigroups. We also refer to \cite[Section 3.2.3]{Acz1}, \cite[Chapter 13]{Acz2} and \cite[Chapter 4]{ST1} for further contextual and historical discussions.\par 
The purpose of this paper is to solve the functional equations \eqref{E3}, \eqref{E1} and \eqref{E2} on a semigroup $S$, where $\sigma :S\rightarrow S$ is an automorphism not necessarily involutive, our results are natural extensions of previous results about the solutions of \eqref{E3}, \eqref{E1} and \eqref{E2}.\par 
The character involutive of the automorphism $\sigma$ is used in the proofs of \cite[Theorem 4.12, Theorem 4.16]{ST1}, \cite[Theorem 4.1, Corollary 4.3]{EB1} and \cite[Theorem 4.2, Theorem 5.1]{Ase2}. The present paper shows that this condition is not crucial for the giving proofs even in the setting of semigroups.\par 
The contributions of the present work to the knowledge about solutions of
\eqref{E3}, \eqref{E1} and \eqref{E2} are the following:\\
(1) By the help of Theorem \ref{phi1} and Theorem \ref{phi2} (See section 2) we find all complex-valued solutions of \eqref{E3}, \eqref{E1} and \eqref{E2} on a semigroup $S$ in terms of multiplicative functions and solutions $\phi$ of the special sine addition law
\begin{equation}
\phi(xy)=\phi(x)\chi(y)+\phi(y)\chi(x),\ \ x,y\in S,
\label{spsine}
\end{equation} 
where $\chi:S\rightarrow\mathbb{C}$ is a multiplicative function. The new result here is that the automorphism $\sigma$ is not involutive. That makes the exposition more involved and explains why our proofs are longer than those of previous papers about the same functional equations.\\
(2) We consider the variants  of \eqref{E3} and \eqref{E1} respectively
\begin{equation}
g(\sigma(y)x)=g(x)g(y)+f(x)f(y),\quad x,y\in S,
\label{var1}
\end{equation}
\begin{equation}
f(\sigma(y)x)=f(x)g(y)+f(y)g(x),\quad x,y\in S.
\label{var2}
\end{equation}
We show that \eqref{var1} is equivalent to \eqref{E3}, and \eqref{var2} is equivalent to \eqref{E1}.\\
(3) As an application, we determine the complex-valued solutions of the new functional equations
\[g(x+\beta y)=g(x)g(y)+f(x)f(y),\quad x,y\in \mathbb{R},\]
\[f(x+\beta y)=f(x)g(y)+f(y)g(x),\quad x,y\in \mathbb{R},\] 
where $\beta \in \mathbb{R}\backslash \lbrace 0,-1,1\rbrace$. Obviously these equations generalizes the
functional equations
\[g(x\pm y)=g(x)g(y)+f(x)f(y),\quad x,y\in \mathbb{R},\]
\[f(x\pm y)=f(x)g(y)+f(y)g(x),\quad x,y\in \mathbb{R},\]
which have been studied by many authors (See for example \cite[Corollary 3.56.a, Corollary 3.56.c]{K} and \cite[Corollary 4.17]{ST1}).\par 
The outline of the paper is as follows: In the next section we give some notations and terminology. The complete solution of \eqref{E3} is given in section 3. In section 4 we solve the functional equation \eqref{E1}. The sine subtraction formula \eqref{E2} is solved in section 5. Section 6 contains some applications.
\section{Notations and terminology}
In this section we give some notations and notions that are essential in our discussion. Throughout this paper $S$ denotes a semigroup. That is a set equipped with an associative binary operation. A  multiplicative function on $S$ is a function $\mu :S\rightarrow \mathbb{C}$ satisfying $\mu(xy) = \mu(x)\mu(y)$ for all $x, y \in S$. A function $f:S\rightarrow \mathbb{C}$ is central if $f(xy) = f(yx)$ for all $x, y\in S$, and $f$ is abelian if $f$ is central and $f(xyz)=f(xzy)$ for all $x,y,z\in S$. We define the set $S^2:=\lbrace xy\  \vert \  x, y\in S\rbrace$. Let $\sigma: S \rightarrow S$ be an automorphism, for any function $f :S \rightarrow \mathbb{C}$ we define the function $f^*:=f\circ \sigma$. A topological semigroup is a pair, consisting of a semigroup $S$ and a topology on $S$, such that the product $(x,y)\mapsto xy$ is continuous from $S\times S$ to $S$, when $S$ is given the product topology. If $S$ is a topological semigroup, let $C(S)$ denote the set of continuous functions mapping $S$ into $\mathbb{C}$.
  \begin{nota}
  Let $\chi$ be a non-zero multiplicative function. The sympbol $\phi_{\chi}$ shall denote a solution of the sepcial sine addition law \eqref{spsine}. i.e
  $$\phi_{\chi}(xy)=\phi_{\chi}(x)\chi(y)+\phi_{\chi}(y)\chi(x),\ \ x,y\in S.$$
  \label{Nota1}
  \end{nota}
  The following are respectively  \cite[Theorem 3.2]{EB1} and \cite[Theorem 3.1]{EB2}, but for brivety some formulas of solution are expressed with the use of Notation \ref{Nota1}.
  \begin{thm}
  The solutions $g,f:S\rightarrow\mathbb{C}$ of the functional equation 
  $$g(xy)=g(x)g(y)-f(x)f(y),\ \ x,y\in S,$$
  are the following pairs:
  \item[(1)] $g=f=0$.
  \item[(2)] $g=\dfrac{\delta ^{-1}\chi_1+\delta \chi_2}{\delta^{-1}+\delta}$ and $f=\dfrac{\chi_1-\chi_2}{i\left(\delta^{-1}+\delta \right) }$, where $\chi_1,\chi_2:S\rightarrow\mathbb{C}$ are two different multiplicative functions and $\delta \in \mathbb{C}\backslash \left\lbrace 0,i,-i \right\rbrace $.
  \item[(3)] $f$ is any non-zero function such that $f=0$ on $S^2$ and $g=\pm f$.
  \item[(4)] $g=\chi\pm \phi_{\chi}$ and $f=\phi_{\chi}$, where $\chi:S\rightarrow\mathbb{C}$ is a non-zero multiplicative function.
  \label{phi1}
  \end{thm}
  \begin{thm}
  The solutions $g,f:S\rightarrow\mathbb{C}$ of the functional equation 
  $$f(xy)=f(x)g(y)+f(y)g(x),\ \ x,y\in S,$$
  with $f\neq 0$ can be listed as follows:
  \item[(1)] $f=c\left( \chi_1-\chi_2\right) $ and $g=\dfrac{\chi_1+\chi_2}{2}$, where $\chi_1,\chi_2:S\rightarrow\mathbb{C}$ are two different multiplicative functions and $c\in \mathbb{C}\backslash \left\lbrace 0 \right\rbrace $.
  \item[(2)] $f$ is any non-zero function such that $f=0$ on $S^2$ and $g=0$.
  \item[(3)] $f=\phi_{\chi}$ and $g=\chi$, where $\chi:S\rightarrow\mathbb{C}$ is a non-zero multiplicative function.
  \label{phi2}
  \end{thm}
  The following lemma will be used throughout the paper without explicit mentionning.
\begin{lem}
Let $f:S\rightarrow \mathbb{C}$ be a non-zero function satisfying
\begin{equation}
f(x\sigma(y))=\beta f(x)f(y),\quad\text{for all}\quad x,y\in S,
\label{M1}
\end{equation} 
where $\beta \in \mathbb{C}\backslash \lbrace 0\rbrace$ is a constant. Then there exists a non-zero multiplicative function $\chi:S\rightarrow \mathbb{C}$ such that $\beta f =\chi$ and $\chi^*=\chi$.
\label{M}
\end{lem}
\begin{proof}
By using the associativity of the semigroup operation we compute $f(x\sigma(y)\sigma(z))$ using the identity \eqref{M1} first as $f((x\sigma(y))\sigma(z))$ and then as $f(x(\sigma(y)\sigma(z)))$ and compare the results to obtain 
\begin{equation}
\beta^2 f(x)f(y)f(z)=\beta f(x)f(yz),\quad \text{for all}\quad x,y,z\in S.
\label{M2}
\end{equation}
Since $f\neq 0$ and $\beta \neq 0$ Eq. \eqref{M2} can be written as 
\begin{equation}
f(yz)=\beta f(y)f(z), \quad \text{for all}\quad y,z\in S.
\label{M3}
\end{equation}
This implies that the function $\chi:=\beta f$ is multiplicative. That is $f=\dfrac{1}{\beta}\chi$, and then Eq. \eqref{M1} becomes
\begin{equation}
\dfrac{1}{\beta}\chi(x)\chi^*(y)=\dfrac{1}{\beta}\chi(x)\chi(y).
\label{M4}
\end{equation}
Since $\chi\neq 0$ and $\beta \neq 0$, we deduce from \eqref{M4} that $\chi^*=\chi$. This completes the proof of Lemma \ref{M}.
\end{proof} 
\section{The cosine subtraction formula \eqref{E3}}
The most recent result on the cosine subtraction formula \eqref{E3}, namely
\[g(x\sigma(y)=g(x)g(y)+f(x)f(y),\quad x,y\in S,\]
on semigroups is \cite[Theorem 4.2]{Ase2}. By using similar computations to those of \cite[Theorem 4.2]{Ase2} we will solve \eqref{E3} on general semigroups, but here $\sigma$ is not assumed to be involutive. It should be mentioned that in the case of an involutive automorphism $\sigma$, the general solution of \eqref{E3} on monoids can be found in \cite[Theorem 4.1]{EB1}.
\begin{lem}
Let $f,g:S\rightarrow\mathbb{C}$ be a solution of Eq. \eqref{E3}, and suppose that $f$ and $g$ are linearly independent. Then $g^*=g$ and $f^*=f$ or $f^*=-f$.
\label{Le1}
\end{lem}
\begin{proof}
By using the associativity of the semigroup operation we compute $g(x\sigma(y)\sigma(z))$ by the help of Eq. \eqref{E3} first as $g((x\sigma(y))\sigma(z))$ and then as $g(x(\sigma(y)\sigma(z)))$ and compare the results. We obtain after some rearrangement that
\begin{equation}
f(x)\left[f(yz)-f(y)g(z) \right]+g(x)\left[ g(yz)-g(y)g(z)\right]=f(z)f(x\sigma(y)).
\label{A4}  
\end{equation}
Since $f\neq 0$, there exists $z_0\in S$ such that $f(z_0)\neq 0$ and then 
\begin{equation}
f(x)h(y)+g(x)k(y)=f(x\sigma(y)),
\label{A5}
\end{equation}
where 
$$h(y)=\dfrac{f(yz_0)-f(y)g(z_0)}{f(z_0)},$$
and $$k(y)=\dfrac{g(yz_0)-g(y)g(z_0)}{f(z_0)}.$$
By using \eqref{E3} and the fact that $\sigma$ is a bijection, we obtain
\begin{equation}
k=c_1g+c_2f,
\label{A6}
\end{equation}
for some constants $c_1,c_2\in \mathbb{C}$. Substituting  \eqref{A5} into \eqref{A4}, we obtain 
\begin{align}
\begin{split}
f(x)\left[f(yz)-f(y)g(z) \right]+g(x)\left[ g(yz)-g(y)g(z)\right]\\= f(x)f(z)h(y)+g(x)f(z)k(y).
\label{A7}
\end{split}
\end{align}
Since $f$ and $g$ are linearly independent we deduce from \eqref{A7} that
\begin{equation}
g(yz)=g(y)g(z)+f(z)k(y),
\label{A8}
\end{equation}
and 
\begin{equation}
f(yz)=f(y)g(z)+f(z)h(y).
\label{A9}
\end{equation}
Substituting \eqref{A6} into \eqref{A8}, we get
\begin{equation}
g(yz)=\left[ g(z)+c_1f(z)\right] g(y)+c_2f(z)f(y).
\label{A10}
\end{equation}
So, by applying \eqref{A10} to the pair $(y,\sigma (z))$, we obtain 
\begin{equation}
g(y\sigma(z))=\left[ g^*(z)+c_1f^*(z)\right] g(y)+c_2f^*(z)f(y).
\label{A11}
\end{equation}
By comparing \eqref{A10} and \eqref{E3}, and using the linear independence of $f$ and $g$, we get 
\begin{equation}
g=g^*+c_1f^*,
\label{A12}
\end{equation}
\begin{equation}
f=c_2f^*.
\label{A13}
\end{equation}
Since $f\neq 0$, we deduce from \eqref{A13} that $c_2\neq 0$.\\
\underline{First case :} $c_1=0$. That is $g^*=g$.  By applying Eq. \eqref{E3} to the pair $(\sigma(x),y)$ we find that 
\[g^*(xy)=g^*(x)g(y)+f^*(x)f(y).\]
So since $f^*=\dfrac{1}{c_2}f$ and $g^*=g$, we get that for all $x,y\in S$
\begin{equation}
g(xy)=g(x)g(y)+\dfrac{1}{c_2}f(x)f(y).
\label{AA1}
\end{equation}
Now if we apply Eq. \eqref{AA1} to the pair $(x,\sigma(y))$, we get
\begin{equation}
g(x\sigma(y))=g(x)g(y)+\dfrac{1}{c_2^2}f(x)f(y).
\label{AA2}
\end{equation}
Comparing \eqref{E3} and \eqref{AA2} and using that $f\neq 0$, we deduce that $c_2^2=1$. So $c_2=\pm 1$, and then $f=f^*$ or $f=-f^*$.\\
\underline{Second case :} $c_1\neq 0$. We get from \eqref{A12} and \eqref{A13} that $g^*=g-\dfrac{c_1}{c_2}f$, and then by applying Eq. \eqref{E3} to the pair $(\sigma(x),y)$ we get
\begin{equation}
g(xy)-\dfrac{c_1}{c_2}f(xy)=g(x)g(y)-\dfrac{c_1}{c_2}f(x)g(y)+\dfrac{1}{c_2}f(x)f(y).
\label{AA3}
\end{equation}
Then by using Eq. \eqref{A10}, we get from \eqref{AA3} after some rearrangement that 
\begin{equation}
f(xy)=f(y)\left(c_2g(x)+\left( \dfrac{c_2^2-1}{c_1}\right) f(x) \right) +g(y)f(x).
\label{AA4}
\end{equation}
Comparing Eq. \eqref{A9} and Eq. \eqref{AA4}, and using that $f\neq 0$ we deduce that
\begin{equation}
h=c_2g+\left( \dfrac{c_2^2-1}{c_1}\right) f.
\label{AA5}
\end{equation} 
Taking \eqref{AA5} and \eqref{A6} into account Eq. \eqref{A5} becomes
\begin{equation}
f(x\sigma(y))=g(x)\left(c_2f(y)+c_1g(y) \right) +f(x)\left(\left( \dfrac{c_2^2-1}{c_1}\right) f(y)+c_2g(y) \right).
\label{AA6} 
\end{equation}
Now, if we apply Eq. \eqref{AA4} to the pair $(x,\sigma(y))$ we get
\begin{equation}
f(x\sigma(y))=f(x)\left(g(y)+\left( \dfrac{c_2^2-c_1^2-1}{c_1c_2}\right) f(y) \right) +g(x)f(y).
\label{AA7}
\end{equation}
Comparing Eq. \eqref{AA6} and Eq. \eqref{AA7} and using the linear independence of $f$ and $g$, we get 
$$g=\left( \dfrac{1-c_2}{c_1}\right) f.$$
This is a contradiction since $f$ and $g$ are linearly independent. So this case does not occur. This completes the proof of Lemma \ref{Le1}.
\end{proof}
\begin{rem}
The result of Lemma \ref{Le1} is also true for the variant \eqref{var1} of  equation \eqref{E3}. 
\begin{proof}
Let $f,g:S\rightarrow\mathbb{C}$ be a solution of Eq. \eqref{var1} such that $f$ and $g$ are linearly independent, if we compute $g(\sigma(yz)x)$ in two different ways we get by using the linear independent of $f$ and $g$ that 
\begin{equation}
f(\sigma(z)x)=f(x)h(z)+g(x)\left[a_1g(z)+a_2f(z) \right], 
\label{rev1}
\end{equation}
\begin{equation}
f(yz)=g(y)f(z)+f(y)h(z), 
\label{rev2}
\end{equation}
\begin{equation}
g(yz)=g(z)\left[g(y)+a_1f(y) \right]+a_2f(y)f(z) , 
\label{rev3}
\end{equation}
for some constants $a_1,a_2\in \mathbb{C}$ and $h$ is a function. If we apply Eq. \eqref{rev3} to $(\sigma(y),z)$ and compare the preceding equation with Eq. \eqref{var1} we obtain since $f$ and $g$ are linearly independent that $g=g^*+a_1f^*$ and $f=a_2f^*$. $f\neq 0$ implies that $a_2\neq 0$.\\
\underline{First case :} $a_1=0$. In this case $g^*=g$, so if we apply Eq. \eqref{var1} to $(\sigma(x),y)$ and then apply the preceding equation to $(x,\sigma(y))$ we obtain 
\begin{equation}
g(\sigma(y)x)=g(x)g(y)+\dfrac{1}{a_2^2}f(x)f(y).
\label{rev4}
\end{equation}
Comparing Eq. \eqref{rev4} with \eqref{var1} and using that $f\neq 0$  we deduce that $a_2^2=1$. That is $a_2=\pm 1$, so $f=f^*$ or $f=-f^*$.\\
\underline{Second case :} $a_1\neq 0$. If we apply Eq. \eqref{var1} to $(\sigma(x),y)$ and using Eq. \eqref{rev3} we find that 
\begin{equation}
f(yx)=f(y)\left(a_2g(x)+\left( \dfrac{a_2^2-1}{a_1}\right) f(x) \right) +g(y)f(x).
\label{rev5}
\end{equation}
Comparing Eq. \eqref{rev2} and Eq. \eqref{rev5} and using that $f\neq 0$, we get that $h=a_2g+\dfrac{a_2^2-1}{a_1}f$. Now Eq. \eqref{rev1} becomes
\begin{equation}
f(\sigma(y)x)=g(x)\left(a_2f(y)+a_1g(y) \right) +f(x)\left(\left( \dfrac{a_2^2-1}{a_1}\right) f(y)+c_2g(y) \right).
\label{rev6}
\end{equation}
By applying Eq. \eqref{rev5} to the pair $(x,\sigma(y))$, we find that 
\begin{equation}
f(\sigma(y)x)=f(x)\left(g(y)+\left( \dfrac{a_2^2-a_1^2-1}{a_1a_2}\right) f(y) \right) +g(x)f(y).
\label{rev7}
\end{equation}
Comparing these last two identities and using the linear independence of $f$ and $g$, we get $g=\dfrac{1-a_2}{a_1}f$ but this is a contradiction since $f$ and $g$ are linearly independent. This completes the proof of Remark \ref{RR1}.
\end{proof}

 \label{RR1}
\end{rem}
The next result gives the general solution of \eqref{E3} on semigroups.
\begin{thm}
The solutions $f,g:S\rightarrow\mathbb{C}$ of the functional equation \eqref{E3} are the following :
\item[(1)] $g=0$ and $f=0$.
\item[(2)] $g$ is any non-zero function such that $g=0$ on $S^2$, and $f=cg$, where $c\in \lbrace i,-i\rbrace$.
\item[(3)] $g=\dfrac{1}{1+\alpha^2}\chi$ and $f=\dfrac{\alpha}{1+\alpha^2}\chi$, where $\alpha \in \mathbb{C}\backslash \lbrace i,-i\rbrace$ is a constant and $\chi :S\rightarrow\mathbb{C}$ is a non-zero multiplicative function such that $\chi^*=\chi$.
\item[(4)] $g=\dfrac{\delta ^{-1}\chi_1+\delta \chi_2}{\delta^{-1}+\delta}$ and $f=\dfrac{\chi_2-\chi_1}{\delta^{-1}+\delta}$, where $\delta \in \mathbb{C}\backslash \lbrace 0,i,-i\rbrace$  and $\chi_1,\chi_2:S\rightarrow\mathbb{C}$ are two different multiplicative functions such that $\chi_1^* = \chi_1$ and $\chi_2^*=\chi_2$.
\item[(5)] $g=\dfrac{\chi+\chi^*}{2}$ and $f=\dfrac{\chi-\chi^*}{2i}$, where $\chi:S\rightarrow\mathbb{C}$ is a multiplicative function such that $\chi^*\neq \chi$ and $\chi\circ \sigma^2=\chi$.
\item[(6)] $f=-i\phi_{\chi}$ and $g=\chi \pm \phi_{\chi}$
where $\chi:S\rightarrow\mathbb{C}$ is a non-zero multiplicative function such that $\chi^*=\chi$ and $\phi_{\chi}^*=\phi_{\chi}$.\\ 
Note that $f$ and $g$ are Abelian in each case.\par 
Furthermore, if $S$ is a topological semigroup and $g\in C(S)$, then\\ $f,\chi,\chi^*,\chi_1,\chi_2,\phi_{\chi}\in C(S)$.
\label{TE3}
\end{thm}
\begin{proof}
If $g=0$ then $f=0$. This is case (1). So from now on we assume that $g\neq 0$. Suppose that $g=0$ on $S^2$, then we get from equation \eqref{E3} that 
\begin{equation}
g(x)g(y)+f(x)f(y)=0.
\label{A1}
\end{equation}
Since $g\neq 0$ we obtain from equation \eqref{A1} that $f=cg$ where $c\in \mathbb{C}$ is a constant. Then if we take this into account in equation \eqref{A1} we get that $(c^2+1)g(x)g(y)=0$. This implies that $c^2+1=0$ because $g\neq 0$, so $c\in \lbrace i,-i\rbrace$. This occurs in part (2) of Theorem \ref{TE3}. If $f=0$, then equation \eqref{E3} can be written as follows $g(x\sigma(y))=g(x)g(y)$. So $g=:\chi$ is multiplicative and $\chi^*=\chi$. This occurs in part (3) of Theorem \ref{TE3} with $\alpha=0$. Now we assume that $g\neq 0$ on $S^2$, $f\neq 0$ and we  discuss two cases according to whether $f$ and  $g$ are linearly dependent or not.\\
\underline{First case :} $g$ and $f$ are linearly dependent. There exists a constant $\alpha \in \mathbb{C}$ such that $f=\alpha g$, so equation \eqref{E3} can be written as 
\begin{equation}
g(x\sigma(y))=(1+\alpha^2) g(x)g(y),\quad x,y\in S.
\label{A3}
\end{equation}
Since $g\neq 0$ on $S^2$ and $f\neq 0$, we deduce from \eqref{A3} that $\alpha \notin \lbrace 0,i,-i\rbrace$, and then $\chi:=(1+\alpha^2) g$ is multiplicative and $\chi^*=\chi$. This occurs in case (3) with $\alpha \neq 0$.\\
\underline{Second case :} $g$ and $f$ are linearly independent. According to Lemma \ref{Le1} $g^*=g$ and $f=f^*$ or $f^*=-f$.\\
\underline{Subcase A :} $f=f^*$.  By applying Eq. \eqref{E3} to the pair $(\sigma(x),y)$, we obtain \begin{equation}
g(xy)=g(x)g(y)+f(x)f(y),\quad x,y\in S.
\label{A14}
\end{equation}
Defining $l:=if$, equation \eqref{A14} can be written as follows
$$g(xy)=g(x)g(y)-l(x)l(y),\quad x,y\in S.$$
According to Theorem \ref{phi1} and taking into account that $f$ and $g$ are linearly independent, we have the following possibilities :\\
(i) $g=\dfrac{\delta ^{-1}\chi_1+\delta \chi_2}{\delta^{-1}+\delta}$ and $l=\dfrac{\chi_1-\chi_2}{i(\delta^{-1}+\delta)}$, where $\delta \in \mathbb{C}\backslash \lbrace 0,i,-i\rbrace$ is a constant and $\chi_1,\chi_2:S\rightarrow\mathbb{C}$ are two multiplicative functions such that $\chi_1\neq \chi_2$. Since $g=g^*$, $f=f^*$ and $l=if$, we deduce that $f=\dfrac{\chi_2-\chi_1}{\delta^{-1}+\delta}$, $\chi_1=\chi_1^*$ and $\chi_2=\chi_2^*$. This is case (4).\\
(ii) $g=\chi \pm l$ and  $l=\phi_{\chi}$, where $\chi:S\rightarrow\mathbb{C}$ is a non-zero  multiplicative function. Since $f^*=f$ and $g^*=g$, we see that  $\chi^*=\chi$ and $\phi_{\chi}^*=\phi_{\chi}$. In addition $l=if$ implies that $f=-i\phi_{\chi}$. This occurs in part (6).\\
\underline{Subcase B :} $f^*=-f$. Equation \eqref{E3} can be written as follows
$$g(xy)=g(x)g(y)-f(x)f(y),\quad x,y\in S.$$
Similarly to the previous case, we get according to Theorem \ref{phi1} and taking into account that $f$ and $g$ are linearly independent the two cases:\\
(i) $g=\dfrac{\delta ^{-1}\chi_1+\delta \chi_2}{\delta^{-1}+\delta}$ and $f=\dfrac{\chi_1-\chi_2}{i(\delta^{-1}+\delta)}$, where $\delta \in \mathbb{C}\backslash \lbrace 0,i,-i\rbrace$ is a constant and $\chi_1,\chi_2:S\rightarrow\mathbb{C}$ are two different multiplicative functions. Since $f^*=-f$ and $g^*=g$, we get 
\begin{equation}
\delta ^{-1}(\chi_1-\chi_1^*)+\delta (\chi_2-\chi_2^*)=0,
\label{Par1}
\end{equation}
\begin{equation}
\chi_1+\chi_1^*=\chi_2+\chi_2^*.
\label{Par2}
\end{equation}
Since $\chi_1\neq\chi_2$, we obtain by the help of \cite[Corollary 3.19]{ST1} that $\chi_1=\chi_2^*$ and $\chi_2=\chi_1^*$.  Then \eqref{Par1} reduces to 
$$\left(\delta^{-1}-\delta \right) \left(\chi_1-\chi_1^* \right)=0.$$ 
This implies that $\delta^{-1}-\delta=0$ since $\chi_1\neq\chi_2$. That is $\delta=\pm 1$. This occurs in case (5) with $\chi_1=\chi$ and $\chi_2=\chi^*$. In addition $\chi_1=\chi_2^*$ implies that $\chi\circ\sigma^2=\chi$.\\
(ii) $g=\chi \pm f$ and  $f=\phi_{\chi}$ where $\chi:S\rightarrow\mathbb{C}$ is a non-zero multiplicative function. If $g=\chi+ f$ we obtain since $f^*=-f$ and $g^*=g$ that $g=\chi^*- f$. Adding and subtracting this from $g=\chi +f$ we get that 
\[g=\dfrac{\chi+\chi^*}{2}\quad\text{and}\quad f=\dfrac{\chi^*-\chi}{2}.\]
By assumption $f\neq 0$, so $\chi\neq \chi^*$. By substituting the forms of $f$ and $g$ in \eqref{E3} we find that $\chi=\chi^*$. This case does not occur. Now if $g=\chi-f$, we show by the same way that $g=\dfrac{\chi+\chi^*}{2}$ and $f=\dfrac{\chi-\chi^*}{2}$ which leads by substitution to $\chi=\chi^*$ ($f=0$). This case does not occur.\par
Conversely we check by elementary computations that the forms (1), (2), (3), (4), (5) and (6) satisfy \eqref{E3}.\par 
 Finally, suppose that $S$ is a topological semigroup and $g\in C(S)$. In case (1), $f=0\in C(S)$. In case (2), $f=\pm ig\in C(S)$. Now if $f\neq 0$, the continuity of $f$ follows easily from the continuity of $g$ and the functional equation \eqref{E3}. Let $y_0\in S$ such that $f(y_0)\neq 0$, we get from \eqref{E3} that 
\[f(x)=\dfrac{g(x\sigma(y_0))-g(y_0)g(x)}{f(y_0)}\ \text{for}\ x\in S.\] 
The function $x\mapsto g(x\sigma(y_0))$ is continuous, since the right translation $x\mapsto x\sigma(y_0)$ from $S$ into $S$ is continuous, so $f$ is continuous as a linear combination of continuous functions. In cases (4) and (5) we get the continuity of $\chi_1,\chi_2,\chi,\chi^*$ by the help of \cite[Theorem 3.18]{ST1}. In case (6), $\phi_{\chi}=if\in C(S)$ and $\chi=(g\pm \phi_{\chi})\in C(S)$. This completes the proof of Theorem \ref{TE3}.
\end{proof}
Now we relate the solution of the variant \eqref{var1} to the functional equation \eqref{E3}. 
\begin{prop}
The functional equation \eqref{E3} and its variant \eqref{var1}, namely
\[g(\sigma(y)x)=g(x)g(y)+f(x)f(y),\quad x,y\in S,\]
 have the same solutions.
 \label{pp1}
 \end{prop}
 \begin{proof}
 Theorem \ref{TE3} shows that if $(g,f)$ is a solution of \eqref{E3}, then $g$ is abelian, in particular central. So any solution $(g,f)$ of \eqref{E3} is also a solution of \eqref{var1}. Now let $(g,f)$ be a solution of \eqref{var1}. We show that $g$ is central.\\
 \underline{First case:} $g$ and $f$ are linearly dependent. That is $f=\delta g$ for some constant $\delta \in \mathbb{C}$. Eq. \eqref{var1} can be written as 
 \[g(\sigma(y)x)=(1+\delta^2)g(x)g(y),\quad x,y\in S.\]
 If $\delta \in \left\lbrace i,-i \right\rbrace $, then $g=0$ on $S^2$, so $g$ is central and we are done. If $\delta \neq \pm i$, we get that $(1+\delta^2)g$ is multiplicative. Then $g$ is central.\\
 \underline{Second case:} $g$ and $f$ are linearly independent. By Remark \ref{RR1} we have $g=g^*$ and $f=f^*$ or $f=-f^*$. By applying Eq.\eqref{var1} to $(\sigma(x),y)$ we obtain 
 \[g(yx)=g(x)g(y)\pm f(x)f(y),\quad x,y\in S.\]
 This implies that $g$ is central. This completes the proof of Proposition \ref{pp1}.
 \end{proof}
\section{The sine addition formula \eqref{E1}}
In this section we solve the functional equation \eqref{E1} on semigroups. The following lemma gives some key properties of the solutions of Eq. \eqref{E1}. We use similar computations to those of \cite[Theorem 5.1]{Ase2} but of course here $\sigma$ is not involutive.
 \begin{lem}
 Suppose $f,g:S\rightarrow\mathbb{C}$ satisfy Eq. \eqref{E1} such that $f$ and $g$ are linearly independent. Then $f^*=f$ and $g^*=g$.
 \label{Le2}
 \end{lem}
 \begin{proof}
 Computing $f(x\sigma(yz))$  in two different ways using equation \eqref{E1}, we obtain after some rearrangement that
\begin{equation}
f(x)\left[g(yz)-g(y)g(z) \right]+g(x)\left[ f(yz)-f(y)g(z)\right]=f(z)g(x\sigma(y)).
\label{B6}  
\end{equation}
Since $f\neq 0$, there exists $z_0\in S$ such that $f(z_0)\neq 0$ and hence 
\begin{equation}
f(x)h(y)+g(x)k(y)=g(x\sigma(y)),
\label{B7}
\end{equation}
where 
$$h(y)=\dfrac{g(yz_0)-g(y)g(z_0)}{f(z_0)},$$
and $$k(y)=\dfrac{f(yz_0)-f(y)g(z_0)}{f(z_0)}.$$
By using Eq. \eqref{E1} and the fact that $\sigma$ is a bijection, we get
\begin{equation}
k=\alpha f+\beta g,
\label{B10}
\end{equation}
for some constants $\alpha,\beta \in \mathbb{C}$. Now by using \eqref{B7}, equation \eqref{B6} becomes
\begin{align}
\begin{split}
f(x)\left[g(yz)-g(y)g(z) \right]+g(x)\left[ f(yz)-f(y)g(z)\right]\\= f(x)f(z)h(y)+g(x)f(z)k(y).
\label{B8}
\end{split}
\end{align}
Since $f$ and $g$ are linearly independent we deduce from \eqref{B8} that for all $y,z\in S$
\begin{equation}
g(yz)=g(y)g(z)+f(z)h(y),
\label{D1}
\end{equation}
and 
\begin{equation}
f(yz)=f(y)g(z)+f(z)k(y).
\label{B9}
\end{equation}
By using \eqref{B10}, equation \eqref{B9} can be written as follows
$$f(yz)=\left(g(z)+\alpha f(z) \right)f(y)+\beta f(z)g(y).$$
This implies that 
$$f(y\sigma(z))=\left(g^*(z)+\alpha f^*(z) \right)f(y)+\beta f^*(z)g(y).$$
By comparing this last identitie with Eq. \eqref{E1} and using the linear independence of $f$ and $g$ we deduce that
\begin{equation}
g=g^*+\alpha f^*,
\label{B11}
\end{equation}
\begin{equation}
f=\beta f^*.
\label{B12}
\end{equation}
Since $f\neq 0$ we get from \eqref{B12} that $\beta \neq 0$, and from \eqref{B11} that $g^*=g-\dfrac{\alpha}{\beta}f$.\\ 
So, for all $x,y\in S$ we have 
\begin{align*}
f(x\sigma(y)) &= \beta f^*(x\sigma(y))=\beta f(\sigma(x)\sigma(\sigma(y))) \\
& = \beta f^*(x)g^*(y)+\beta f^*(y)g^*(x) \\
& = f(x)\left[g(y)-\dfrac{\alpha }{\beta}f(y) \right]+f(y)\left[g(x)-\dfrac{\alpha}{\beta} f(x) \right] \\
&=f(x)g(y)+f(y)g(x)-\dfrac{2\alpha}{\beta} f(x)f(y)\\
&=f(x\sigma(y))-\dfrac{2\alpha}{\beta} f(x)f(y). 
\end{align*}
So $\alpha=0$ since $f\neq 0$, and then $g^*=g$. Now if we apply Eq. \eqref{E1} to the pair $(\sigma(x),y)$ and multiplying the preceding equation by $\beta$ we get
\begin{equation}
f(xy)=f(x)g(y)+\beta f(y)g(x).
\label{BB1}
\end{equation}
Computing $f(xyz)$ in two different ways, we obtain from \eqref{BB1} by using Eq. \eqref{D1} after some rearrangement that
\begin{equation}
\left(\beta^2-\beta \right)g(x)g(y)=\beta f(y)h(x)-f(x)h(y).
\label{BB2} 
\end{equation}
Since $f\neq 0$, we deduce from Eq. \eqref{BB2} that 
\begin{equation}
h=af+bg,
\label{BB3}
\end{equation}
for some constants $a,b\in \mathbb{C}$. Taking Eq. \eqref{BB3} into account Eq. \eqref{BB2} becomes 
\begin{equation}
\left(\beta^2-\beta \right)g(x)g(y)=f(x)\left( (a\beta-a) f(y)-bg(y)\right) +b\beta g(x)f(y).
\label{BB4} 
\end{equation}
Since $f$ and $g$ are linearly independent, we deduce from Eq. \eqref{BB4} that 
$$\left( \beta^2-\beta\right) g=b\beta f.$$
This implies that $\beta=1$ and $b=0$ since $\beta\neq 0$, $f$ and $g$ are linearly independent. So $f=f^*$. This completes the proof of Lemma \ref{Le2}.
 \end{proof}
 \begin{rem}
Similar computations shows that the result of Lemma \ref{Le2} hold for the variant \eqref{var2} of  equation \eqref{E1}.
 \label{RR2}
\end{rem}
 Now we are ready to solve the functional equation \eqref{E1}.
\begin{thm}
The solutions $f,g : S\rightarrow \mathbb{C}$ of Eq. \eqref{E1} are the following pairs:
\item[(1)] $f=0$ and $g$ is arbitrary.
\item[(2)] $f$ is any non-zero function such that $f=0$ on $S^2$, while $g=0$.
\item[(3)] $f=\dfrac{1}{2\alpha}\chi$ and $g=\dfrac{1}{2}\chi$, where $\chi :S\rightarrow \mathbb{C}$ is a non-zero multiplicative function such that $\chi^*=\chi$ and $\alpha\in \mathbb{C}\backslash \lbrace 0\rbrace$.
\item[(4)] $f=c\left(\chi_1-\chi_2 \right) $ and $g=\dfrac{\chi_1+\chi_2}{2}$, where $\chi_1,\chi_2 :S\rightarrow \mathbb{C}$ are two different multiplicative functions such that $\chi_1^*=\chi_1$, $\chi_2^*=\chi_2$ and $c\in \mathbb{C}\backslash \lbrace 0\rbrace$.
\item[(5)] $f=\phi_{\chi}$ and $g=\chi$ where  $\chi: S \rightarrow \mathbb{C}$ is a non-zero multiplicative function such that $\chi^*=\chi$ and $\phi_{\chi}^*=\phi_{\chi}$.\\
Note that, off the exceptional case (1) $f$ and $g$ are Abelian.\par 
Furthermore, off the exceptional case (1),  if $S$ is a topological semigroup and $f\in C(S)$, then $g,\chi$,$\chi_1,\chi_2,\phi_{\chi}\in C(S)$.
\label{P1}
\end{thm}
\begin{proof}
If $f=0$ then $g$ will be arbitrary. This occurs in case (1). From now on we assume that $f\neq 0$. Suppose that $f=0$ on $S^2$. For all $x,y\in S$, we get from equation \eqref{E1} that 
\begin{equation}
f(x)g(y)+f(y)g(x)=0.
\label{B1}
\end{equation}
Since $f\neq 0$ we deduce from equation \eqref{B1} according to \cite[Exercise 1.1(b)]{ST1} that $g=0$. This occurs in part (2) of Theorem \ref{P1}. Now we assume that $f\neq 0$ on $S^2$ and we  discuss two cases according to whether $f$ and  $g$ are linearly dependent or not.\\
\underline{First case :} $f$ and $g$ are linearly dependent. There exists a constant $\alpha \in \mathbb{C}$ such that $g=\alpha f$, so equation \eqref{E1} can be written as follows $f(x\sigma(y))=2\alpha f(x)f(y)$. This implies that $\alpha \neq 0$, since $f\neq 0$ on $S^2$. So the function $\chi:=2\alpha f$ is multiplicative and $\chi^*=\chi$. This is case (3).\\
\underline{Second case :} $f$ and $g$ are linearly independent. According to Lemma \ref{Le2} we have $f=f^*$ and $g=g^*$. So equation \eqref{E1} becomes 
\begin{equation}
f(xy)=f(x)g(y)+f(y)g(x).
\end{equation}
According to Theorem \ref{phi2} and taking into account that $f\neq 0 $, $g\neq 0$, $f^*=f$ and $ g^*=g$ we have the following possibilities :\\
(i) $f=c\left( \chi _1 -\chi _2\right) $ and $g=\dfrac{\chi _1+\chi _2}{2}$, for some constant $c\in \mathbb{C}\backslash \lbrace 0\rbrace$ and $\chi_1, \chi _2 : S\rightarrow \mathbb{C}$ are two different multiplicative functions such that $\chi_1^*=\chi_1$ and $\chi_2^*=\chi_2$. This is case (4).\\
(ii) $f=\phi_{\chi}$ and $g=\chi$ where  $\chi: S \rightarrow \mathbb{C}$ is a non-zero multiplicative function such that $\phi_{\chi}^*=\phi_{\chi}$ and $\chi^*=\chi$. This occurs in part (5) of Theorem \ref{P1}.\par 
 Conversely we check by elementary computations that if $f, g$ have one of the forms (1)--(5) then $(f,g)$ is a solution of equation \eqref{E1}.\par 
 For the continuity statements, the continuity of $g$ follows easily from the continuity of $f$ and the functional equation \eqref{E1}. In case (4) we get the continuity of $\chi_1$ and $\chi_2$ by the help of \cite[Theorem 3.18]{ST1}. This completes the proof of Theorem \ref{P1}.
\end{proof}
At this point of our discussion about solutions of \eqref{E1}, a natural question comes up: Can we derive the solution of the variant \eqref{var2} of \eqref{E1} from Theorem \ref{P1} ? the next result gives a positive answer.
\begin{prop}
The functional equation \eqref{E1} and its variant \eqref{var2}, namely
$$f(\sigma(y)x)=f(x)g(y)+f(y)g(x),\quad x,y\in S,$$
have the same solutions.
\label{PP2}
\end{prop}
\begin{proof}
Theorem \ref{P1} prove that if $(f,g)$ is a solution of Eq. \eqref{E1}, then $f$ is abelian. In particular central. So $(f,g)$ is a solution of the variant \eqref{var2}. Now let $f,g:S\rightarrow\mathbb{C}$ be a solution of \eqref{var2}. It suffices to show that $f$ is central.\\
\underline{First case:} $f$ and $g$ are linearly dependent. There exists a constant $\gamma \in \mathbb{C}$ such that $g=\gamma f$. Equation \eqref{var2} becomes
\[f(\sigma(y)x)=2\gamma f(x)f(y),\quad x,y\in S.\]
If $\gamma =0$, then $f=0$ on $S^2$, so $f$ is central. If $\gamma \neq 0$, then $2\gamma f$ is multiplicative, and then $f$ is central.\\
\underline{Second case:} $f$ and $g$ are linearly independent. According to Ramark \ref{RR2} we have $f=f^*$ and $g=g^*$. If we apply Eq. \eqref{var2} to $(\sigma(x),y)$ we get 
\[f(yx)=f(x)g(y)+f(y)g(x),\quad x,y\in S,\]
which implies that $f$ is central. This completes the proof of Proposition \ref{PP2}.
\end{proof}
\section{The sine subtraction formula \eqref{E2}}
In this section we solve the functional equation \eqref{E2}. The following lemma will be used later.  
\begin{lem}
Let $f,g:S\rightarrow\mathbb{C}$ be a solution of Eq. \eqref{E2} such that $f$ and $g$ are linearly independent. Then $f^*=-f$ and $g^*=g+\beta f$ for some constant $\beta \in \mathbb{C}$.
\label{Le3}
\end{lem}
\begin{proof}
By using similar computations to those of the proof of Lemma \ref{Le2} we find that for all $y,z\in S$
\begin{equation}
g(yz)=g(y)g(z)-f(z)h(y),
\label{s1}
\end{equation} and that 
\begin{align*}
g=g^*+af^*,\\
f=-bf^*,
\end{align*}
for some constants $a,b\in \mathbb{C}$ and some function $h$. Since $f\neq 0$ we can see that $b\neq 0$, and then $g^*=g-af^*=g+\dfrac{a}{b}f$. Choosing $\beta =\dfrac{a}{b}$, we get $g^*=g+\beta f$. Now by letting $x=\sigma(x)$ in \eqref{E2} we get that 
\[f(xy)=f(x)g(y)+bf(y)g(x)+af(x)f(y),\quad x,y\in S.\]
Computing $f(xyz)$ in two different ways and using Eq. \eqref{s1} we get after some simplifications that 
\begin{equation}
g(y)\left((a-ab)f(x)+(b-b^2)g(x) \right)=bf(y)h(x)-f(x)h(y).
\label{s2}
\end{equation}
Since $f\neq 0$ we get from Eq. \eqref{s2} that $h=\delta f+\gamma g$ for some constants $\delta ,\gamma \in \mathbb{C}$. Taking this into account Eq. \eqref{s2} becomes 
\[g(y)\left((a-ab)f(x)+(b-b^2)g(x) \right)=f(y)\left((b\delta -\delta)f(x)+b\gamma g(x) \right)-\gamma g(y)f(x). \]
Since $f$ and $g$ are linearly independent we deduce that 
\[(a-ab)f(x)+(b-b^2)g(x)=-\gamma f(x).\] 
Then $b=1$ since $b\neq 0$. That is $f^*=-f$. This completes the proof of Lemma \ref{Le3}.
\end{proof}
The next theorem generelizes the results about solutions of \eqref{E2} found in \cite[Proposition 3.1]{Pou}, \cite[Theorem 5.1]{Ajb2}, \cite[Proposition 3.2]{Ase1} and \cite[Corollary 4.3]{EB1}.
\begin{thm}
The solutions $f,g:S\rightarrow\mathbb{C}$ of Eq. \eqref{E2} are the following pairs:
\begin{enumerate}
\item[(1)] $f=0$ and $g$ is arbitrary.
\item[(2)] $f$ is any non-zero function such that $f=0$ on $S^2$ and $g=\alpha f$, where $\alpha \in \mathbb{C}$.
\item[(3)] $f=c(\chi-\chi^*)$ and $g=\dfrac{\chi+\chi^*}{2}+c_1\dfrac{\chi-\chi^*}{2}$, where $\chi:S\rightarrow\mathbb{C}$ is a multiplicative function such that $\chi\neq \chi^*$, $\chi\circ \sigma^2=\chi$, $c_1\in \mathbb{C}$ and $c\in \mathbb{C}\backslash \lbrace 0\rbrace$.
\item[(4)] $f=\phi_{\chi}$ and $g=\chi+c_2\phi_{\chi}$, where $\chi:S\rightarrow\mathbb{C}$ is a non-zero multiplicative function such that $\chi^*=\chi$, $\phi_{\chi}^*=-\phi_{\chi}$ and $c_2\in \mathbb{C}$ is a constant.
\end{enumerate}
Note that, off the exceptional case (1) $f$ and $g$ are Abelian.\par 
Moreover, off the exceptional case (1),  if $S$ is a topological semigroup and $f\in C(S)$, then $g,\chi$,$\chi^*,\phi_{\chi}\in C(S)$.
\label{TH3}
\end{thm}
\begin{proof}
If $f=0$ it is easy to see that $g$ is arbitrary. This is cas (1). Now we split the discussion in two cases according to whether $f$ and $g$ are linearly dependent or not.\\
\underline{First case:} $f$ and $g$ are linearly dependent. That is $g=\alpha f$ for some constant $\alpha \in \mathbb{C}$. So equation \eqref{E2} becomes
\[f(x\sigma(y))=\alpha f(x)f(y)-\alpha f(y)f(x)=0,\quad x,y\in S.\]
This implies that $f=0$ on $S^2$. This occurs in case (2).\\
\underline{Second case:} $f$ and $g$ are linearly independent. According to Lemma \ref{Le3}, we have $f^*=-f$ and $g^*=g+\beta f$, where $\beta \in \mathbb{C}$ is a constant. Then if we apply Eq. \eqref{E2} to the pair $(\sigma(x),y)$ we obtain 
\[f(xy)=f(x)g(y)+f(y)g(x)+\beta f(x)f(y),\quad x,y\in S.\]
That is 
\[f(xy)=f(x)\left[ g(y)+\dfrac{\beta}{2}f(y)\right]+f(y)\left[ g(x)+\dfrac{\beta}{2}f(x)\right],\quad x,y\in S.\]
According to Theorem \ref{phi2} and taking into account that $f$ and $g$ are linearly indepenedent, we have the following possibilities:\\
(i) $f=c\left( \chi _1 -\chi _2\right) $ and $g+\dfrac{\beta}{2}f=\dfrac{\chi _1+\chi _2}{2}$, for some constant $c\in \mathbb{C}\backslash \lbrace 0\rbrace$ and $\chi_1, \chi _2 : S\rightarrow \mathbb{C}$ are two  different multiplicative functions. Since $f=-f^*$, we get
\[\chi_1+\chi_1^*=\chi_2^*+\chi_2.\]
This implies that $\chi_1=\chi_2^*$ and $\chi_2=\chi_1^*$. This occurs in case (3) with $\chi=\chi_1$, $\chi^*=\chi_2$ and $c_1=\dfrac{-\beta c}{2}$. In addition $\chi_2^*=\chi_1$ implies that $\chi\circ\sigma^2=\chi$.\\
(ii) $f=\phi_{\chi}$ and $g+\dfrac{\beta}{2}f=\chi$ where  $\chi: S \rightarrow \mathbb{C}$ is a non-zero multiplicative function such that $\phi_{\chi}^*=-\phi_{\chi}$. By applying Eq. \eqref{E2} to the pair $(\sigma(x),y)$ we obtain
\[\phi_{\chi}(xy)=\phi_{\chi}(x)\chi(y)+\phi_{\chi}(y)\chi^*(x).\]
On the other hand we have 
\[\phi_{\chi}(xy)=\phi_{\chi}(x)\chi(y)+\phi_{\chi}(y)\chi(x).\]
Comparing these last two identities we can see that $\chi=\chi^*$ since $f\neq 0$. This occurs in part (4) of Theorem \ref{TH3} with $c_2=\dfrac{-\beta}{2}$.\par 
For the converse we can check easily that the forms (1)--(4) satisfy Eq. \eqref{E2}. Finally, if $S$ is a topological semigroup, the continuity statements are easy to verify. This completes the proof of Theorem \eqref{TH3}.
\end{proof}
In the next section we shall apply our theory to two different types of groups. The first one is abelian and the second one is not.
\section{Applications}
\begin{App}
Let $S=(\mathbb{R}, + )$,  let $\beta \in \mathbb{R}\backslash \lbrace 0\rbrace$ be a fixed element and let $\sigma (x)=\beta x$ for all $x\in \mathbb{R}$. The functional equations \eqref{E3} and \eqref{E1} can be written respectively as follows :
\begin{equation}
g(x+\beta y)=g(x)g(y)+f(x)f(y),\quad x,y\in \mathbb{R},
\label{ex1}
\end{equation}
\begin{equation}
f(x+\beta y)=f(x)g(y)+f(y)g(x),\quad x,y\in \mathbb{R}.
\label{ex2}
\end{equation}
We note that equation \eqref{ex1} with $\beta =-1$ is \cite[Example 4.18]{ST1}, and equation \eqref{ex2} with $\beta=1$ is \cite[Example 4.5]{ST1}. We are interested to determine the solutions of  \eqref{ex1} and \eqref{ex2} when $\beta \in \mathbb{R}\backslash \lbrace 0,-1,1\rbrace$. For this we apply Theorem \ref{TE3} to Eq. \eqref{ex1} and Theorem \ref{P1} to Eq. \eqref{ex2}. Let $\chi:S\rightarrow \mathbb{C}$ be a non-zero multiplicative function such that 
\[\chi(\beta x)=\chi(x),\ \text{for all}\ x\in \mathbb{R}.\]
Since $S$ is a group, then $\chi$ is a character. So we get $\chi \left((\beta -1)x \right)=1 $ for all $x\in \mathbb{R}$. Since $\beta \neq 1$, we obtain $\chi=1$. By the same way we show that the only non-zero multiplicative function $\chi$ satisfying $\chi(\beta^2 x)=\chi(x)$ for all $x\in \mathbb{R}$ is $\chi=1$ because $\beta \neq \pm 1$. So the special sine addition law \eqref{spsine} becomes 
\[\phi (x+y)=\phi(x)+\phi(y),\quad x,y\in \mathbb{R}.\]
That is $\phi$ additive. In addition if $\phi (\beta x)=\phi (x)$ for all $x\in \mathbb{R}$, then $\phi =0$ since $\beta \neq 1$.\par 
The solutions $f,g:S\rightarrow\mathbb{C}$ of Eq. \eqref{ex1} are the following:\\
1) $f=0$ and $g=0$.\\
2) $f=\dfrac{\alpha}{1+\alpha^2}$ and $g=\dfrac{1}{1+\alpha^2}$, where $\alpha \in \mathbb{C}\backslash \lbrace i, -i\rbrace$.\\
3) $f=0$ and $g=1$.\par 
The solutions $f,g:S\rightarrow\mathbb{C}$ of Eq. \eqref{ex2} can be listed as follows:\\
1) $f=0$ and $g$ is arbitrary.\\
2) $f=\dfrac{1}{2\alpha}$ and $g=\dfrac{1}{2}$, where $\alpha \in \mathbb{C}\backslash \lbrace 0\rbrace$.\\
\end{App}
\begin{App}
Let $G$ be the $(ax+b)$--group defined by \newline
\[G:=\left\lbrace \left(\begin{matrix}
   a & b  \\
   0 & 1  \\
\end{matrix} 
 \right)\mid a>0,\quad b\in \mathbb{R}  \right\rbrace ,\]
and let $X=\left(\begin{matrix}
   a & b  \\
   0 & 1  \\
\end{matrix} 
 \right)$ for all $a,b\in \mathbb{R}$ such that $a>0$. We consider the following automorphism on $G$ 
 \[\sigma \left(X \right)=\left(\begin{matrix}
   a & 2023b  \\
   0 & 1  \\
\end{matrix} 
 \right).\]
 So $\sigma$ is not involutive. According to \cite[Example 2.10, Example 3.13]{ST1}, the continuous additive and the non-zero multiplicative functions on $G$ have respectively the forms
 \[A_c\left(X \right)= c\log(a) ,\]
 and 
 \[\chi_{\lambda}\left(X \right)=a^{\lambda},\]
 where $c, \lambda \in \mathbb{C}$. We can see that $\chi_{\lambda}\circ \sigma =\chi_{\lambda}$ and $A_c\circ \sigma =A_c$, and it is well known that the non-zero continuous solution $\phi$ of \eqref{spsine} on the group $G$ will be of the form $\phi =\chi_{\lambda} A_c$.\par
 The non-zero solutions $f,g\in C(G)$ of Eq. \eqref{E3} are the following:\\
1) $f \left(X\right)=\dfrac{\alpha a^{\lambda}}{1+\alpha^2}$ and $g \left(X \right)=\dfrac{a^{\lambda}}{1+\alpha^2}$, where $\alpha \in \mathbb{C}\backslash \left\lbrace 0, i, -i\right\rbrace $ and $\lambda \in \mathbb{C}$.\\
 2) $f \left(X\right)=-ica^{\lambda}\log(a)$ and $g \left(X\right)= a^{\lambda}\pm ca^{\lambda}\log(a)$, 
 where $c \in \mathbb{C}\backslash \left\lbrace 0\right\rbrace $ and $\lambda \in \mathbb{C}$.\par 
 The non-zero solutions $f,g\in C(G)$ of Eq. \eqref{E1} are the following:\\
 1) $f \left(X\right)= \dfrac{a^{\lambda}}{2\alpha}$ and $g \left(X\right)= \dfrac{a^{\lambda}}{2}$,  where $\alpha \in \mathbb{C}\backslash \left\lbrace 0\right\rbrace $ and $\lambda \in \mathbb{C}$.\\
 2) $f \left(X\right)= ca^{\lambda}\log(a)$ and $g \left(X \right)= a^{\lambda}$, where $c \in \mathbb{C}\backslash \left\lbrace 0\right\rbrace $ and $\lambda \in \mathbb{C}$.
\end{App}
\subsection*{Declarations}
\textbf{Ethical Approval} Not Applicable.\\
\\
\textbf{Conmpeting interests} None.\\
\\
\textbf{Author contributions} The authors confirm contribution to the paper as follows: study conception and design: Y. Aserrar, E. Elqorachi; data collection: Y. Aserrar; analysis and interpretation of results: Y. Aserrar, E. Elqorachi; draft manuscript preparation: Y. Aserrar. All authors reviewed the results and approved the final version of the manuscript.\\
\\
\textbf{Funding} None.\\
\\
\textbf{Availability of data and materials} Not applicable.

\end{document}